\documentclass[leqno]{aadmbook}
\usepackage{amsthm}
\usepackage{amsfonts}
\usepackage{amsmath}
\usepackage{graphicx}

\usepackage{amssymb}
\usepackage{dsfont}
\usepackage [latin1]{inputenc}

\begin{document}


\newtheorem{theorem}{Theorem}[section]
\newtheorem{lemma}[theorem]{Lemma}
\newtheorem{proposition}[theorem]{Proposition}
\newtheorem{corollary}[theorem]{Corollary}

\theoremstyle{definition}
\newtheorem{definition}[theorem]{Definition}
\newtheorem{example}[theorem]{Example}
\newtheorem{xca}[theorem]{Exercise}

\theoremstyle{remark}
\newtheorem{remark}[theorem]{Remark}

\oddsidemargin 16.5mm
\evensidemargin 16.5mm

\thispagestyle{plain}

\vspace{5cc}
\begin{center}

{\large\bf  EXPLICIT ESTIMATES FOR THE STIRLING NUMBERS OF THE SECOND KIND
\rule{0mm}{6mm}\renewcommand{\thefootnote}{}
\footnotetext{

\rule{2.4mm}{0mm}Keywords and Phrases. Stirling numbers of the second kind, explicit estimates, geometric distribution, exponential distribution}}

\vspace{1cc}
{\large\it Jos\'{e} A. Adell}

\vspace{1cc}
\parbox{24cc}{{\small



We give explicit estimates for the Stirling numbers of the second kind $S(n,m)$. With a few exceptions, such estimates are asymptotically sharp. The form of these estimates varies according to $m$ lying in the central or non-central regions of $\{1,\ldots ,n\}$. In each case, we use a different probabilistic representation of $S(n,m)$ in terms of well known random variables to show the corresponding results.

}}
\end{center}



\section{Introduction}\label{s1}

The Stirling numbers of the second kind $S(n,m)$ can be defined in various equivalent ways (cf. Comtet  \cite[Chap. 5]{Com1974}). For instance,
\begin{equation}\label{eq1.a}
    z^n=\sum_{m=0}^n S(n,m)(z)_m,\quad z\in \mathds{C},
\end{equation}
where $(z)_m=z(z-1)\cdots (z-m+1)$, $m=1,2,\ldots$ is the falling factorial ($(z)_0=1$) or, via their generating function, as
\begin{equation}\label{eq1.b}
    \dfrac{(e^z-1)^m}{m!}=\sum_{n=m}^\infty S(n,m) \dfrac{z^n}{n!},\quad z\in \mathds{C}.
\end{equation}
Since these numbers play an important role in different branches of mathematics, one can find in the literature many papers devoted to obtain manageable formulae for them. Such formulae date back at least to Jordan \cite{Jor1991}, who obtained
\begin{equation}\label{eq1.c}
    S(n,m)\sim \dfrac{m^n}{m!},\qquad S(n,n-m)\sim\dfrac{n^{2m}}{2^m m!}.
\end{equation}
However, such estimates are acceptable only for $m$ quite small in comparison with $n$. Asymptotic expansions for $S(n,m)$, as $n\to \infty$, according to $m$ lying in the central or non-central regions of $\{1,2,\ldots ,n\}$ have been provided by many authors. In this regard, Sachkov \cite[p. 144]{Sac1997} obtained
\begin{equation*}
    S(n,m)=\dfrac{m^n}{m!}\exp \left ( \left ( \dfrac{n}{m}-m\right )e^{-n/m}\right ) (1+o(1)),\quad m<n/\log n.
\end{equation*}
Hsu \cite{Hsu1948} and Chelluri et al. \cite{CheRicTem2000} gave
\begin{equation*}
    S(n,m)=\dfrac{n^{2(n-m)}}{2^{n-m}(n-m)!}\left (1+O(n^{-1/3}) \right ),\quad n-n^{1/3}\leq m\leq n.
\end{equation*}
Moser and Wyman \cite{MosWym1958} and Chelluri et al. \cite{CheRicTem2000} showed that
\begin{equation}\label{eq1.d}
    S(n,m)=\dfrac{n!(e^R-1)^m}{2R^nm!\sqrt{\pi m R H}} \left ( 1+ O (1/m)\right ), \quad 0< \delta\leq m\leq n-n^{1/3},
\end{equation}
where $R$ is the solution to the equation
\begin{equation*}
    R(1-e^{-R})^{-1}=\dfrac{n}{m},
\end{equation*}
and
\begin{equation*}
    H=e^R(e^R-1-R)/2(e^R-1)^2.
\end{equation*}
More specific formulae for $S(n,n-n^a)$, for $1/2<a<1$, can be found in Louchard \cite{Lou2013}.

It seems that less attention has been paid to obtain explicit upper and lower bounds for $S(n,m)$. As far as we know, Moser and Wyman \cite{MosWym1958} provided, whenever $m=n-o(\sqrt{n})$,

\begin{equation}\label{eq1.f}
    S(n,m)=\binom{n}{m} q^{-(n-m)} \left ( \sum_{k=0}^s A_k (n-m)q^k+E_s\right ),\quad q=\dfrac{2}{n-m},
\end{equation}
where $0\leq s\leq n-m-1$, $A_k(r)$ is a polynomial in $r$ for each $k=1,2,\ldots $, and
\begin{equation*}
    |E_s|\leq \dfrac{(2(n-m)^2/5m)^{s+1}}{1-(2(n-m)^2)/5m}.
\end{equation*}
As upper and lower bounds with no essential restriction on $m$, we mention that Rennie and Dobson \cite{RenDob1969} gave
\begin{equation*}
    \dfrac{1}{2}(m^2+m+2)m^{n-m-1}-1 \leq S(n,m)\leq \dfrac{1}{2}\binom{n}{m}m^{n-m},\quad 1\leq m\leq n-1.
\end{equation*}
More recently, using Stein-Chen Poisson approximation, Arratia and DeSalvo \cite{ArrDeS2017} showed that
\begin{equation}\label{eq1.h}
    \binom{N}{m}e^{-2\mu}(1-e^{2\mu}D_{n,m})\leq S(n,n-m)\leq \binom{N}{n}e^{-2\mu}(1+e^{2\mu}D_{n,m}),
\end{equation}
for $n\geq 3$ and $n\geq m\geq 2$, where
\begin{equation*}
    N=\binom{n}{2},\quad \mu=\dfrac{\binom{m}{2}\binom{n}{3}}{\binom{N}{2}},
\end{equation*}
and $D_{n,m}$ is an error term uniformly bounded by 1 which goes to zero if $m=O(\sqrt{n})$. We emphasize that estimates \eqref{eq1.f} and \eqref{eq1.h} are asymptotically sharp provided that $m=n-o(\sqrt{n})$ and $m=O(\sqrt{n})$, respectively.

The aim of this paper is twofold. On the one hand, to give explicit estimates for $S(n,m)$ in the whole range of $m$. Moreover, we show in Section~5 that our estimates are asymptotically sharp, excepting when $n-m\sim m\log (mr_m)$, with $r_m\sim r\in (0,\infty)$, case in which we only are able to give upper an lower bounds of the same order of magnitude.

On the other hand, to open the door to possible extensions to more general numbers, such as the Comtet numbers of the second kind. In fact, many of the results reviewed above are based on formula \eqref{eq1.b}, which seems difficult to extend to other numbers. However, our starting point is the well known formula
\begin{equation}\label{eq1.i}
    \sum_{j=m}^\infty S(j,m)z^j= \dfrac{z^m}{(1-z)(1-2z)\cdots (1-mz)},\quad z\in \mathds{C},\quad |z|<1/m,
\end{equation}
which has a direct generalization to the Comtet numbers of the second kind (see, for instance, Comtet \cite{Com1972} and El-Desouky et al. \cite{ElDCakShi2017}).

To achieve our results, we give different probabilistic representations of $S(n,m)$, whose usefulness depends on the values of $m$. Such representations, written in terms of sums of independent random variables involving the geometric and the exponential distributions, may be of independent interest. Among them, we highlight the following
\begin{equation}\label{eq1.j}
    S(n,m)=\dfrac{m^n}{m!}P(V_{m-1}\leq n-m),\quad 2\leq m\leq n,
\end{equation}
where $V_{m-1}=X(1/m)+\cdots + X((m-1)/m)$ is a sum of $m-1$ independent random variables such that $X(j/m)$ has the geometric distribution with failure probability $j/m$ (a similar representation in terms of multinomial laws was provided in \cite{AdeCar2021}). Having in mind the first asymptotic formula in \eqref{eq1.c}, representation \eqref{eq1.j} gives explicit and sharp estimates for $S(n,m)$ for small values of $m$ in a very simple way (see Theorem~\ref{th4} in Section~3).

This would be the case for other values of $m$, if the random variable $V_{m-1}$, properly standardized, would satisfy the central limit theorem. Unfortunately, the limiting random variable $V$ is rather involved and does not allow us to give a unified treatment for any value of $m$. For these reasons, we provide different probabilistic representations of $S(n,m)$ according to $m$ lying in the central or non-central regions of $\{1,2,\ldots ,n\}$.

The paper is organized as follows. In the following section, we collect the aforementioned probabilistic representations for $S(n,m)$. In Sections~3 and 4, we estimate $S(n,m)$ for $m$ in the non-central and in the central regions, respectively, whereas Section~5 is devoted to discuss in detail the accuracy of our estimates. The main results are stated in Theorems~\ref{th3}, \ref{th4}, and \ref{th5}.

\section{Probabilistic representations}\label{s2}

Let $\mathds{N}$ be the set of positive integers and $\mathds{N}_0=\mathds{N}\cup \{0\}$. Throughout this paper, we always assume that $n,m\in \mathds{N}$ with $m\leq n$ and $z\in \mathds{C}$. The indicator function of the set $A$ is denoted by $1_A$, whereas $\mathds{E}$ stands for mathematical expectation. Unless otherwise specified, all of the random variables appearing in a same expression are supposed to be mutually independent.

To give suitable probabilistic representations of $S(n,m)$, we consider the following random variables. In the first place, let $T$ be a random variable having the exponential density $\rho(\theta)=e^{-\theta}$, $\theta \geq 0$. It is well known that
\begin{equation}\label{eq2.1}
    \mathds{E}T=\text{Var}(T)=1,\qquad \mathds{E}e^{zT}=\dfrac{1}{1-z},\quad |z|<1.
\end{equation}
If $(T_j)_{1\leq j\leq m}$ is a finite sequence of independent copies of $T$, we denote
\begin{equation}\label{eq2.2}
    S_m=T_1+2T_2+ \cdots +mT_m.
\end{equation}

In the second place, let $X(q)$ be a random variable having the geometric distribution with failure probability $q\in (0,1)$, that is,
\begin{equation}\label{eq2.3}
    P(X(q)=k)=pq^k,\quad k\in \mathds{N}_0,\quad p=1-q.
\end{equation}
It is known that
\begin{equation}\label{eq2.4}
    \mathds{E}X(q)=\dfrac{q}{p},\quad \text{Var}(X(q))=\dfrac{q}{p^2},\quad \mathds{E}z^{X(q)}=\dfrac{p}{1-qz},\quad |z|<1/q.
\end{equation}
If $(X(q_j))_{1\leq j\leq m}$ is a finite sequence of independent random variables such that each $X(q_j)$ has the geometric distribution as in \eqref{eq2.3}, we denote
\begin{equation}\label{eq2.5}
    Y_m=X(q_1)+\cdots +X(q_m).
\end{equation}

The following auxiliary result will be needed.

\begin{lemma}\label{l1}
Let $f:\mathds{N}_0\to \mathds{R}$ be a bounded function and let $a>0$ be such that $aq_j<1$, $j=1,\ldots ,m$. Then,
\begin{equation*}
    \mathds{E}a^{Y_m}f(Y_m)=\mathds{E}a^{Y_m}\mathds{E}f(X(aq_1)+\cdots +X(aq_m)).
\end{equation*}
\end{lemma}

\begin{proof}
Set $Y_m=Y_{m-1}+X(q_m)$. By \eqref{eq2.4} and the independence of the random variables involved, we have
\begin{equation*}
    \begin{split}
       & \dfrac{\mathds{E}a^{Y_m}f(Y_m)}{\mathds{E}a^{Y_m}}=\dfrac{1}{\mathds{E}a^{Y_{m-1}}\mathds{E}a^{X(q_m)}}\mathds{E}a^{Y_{m-1}+X(q_m)}f(Y_{m-1}+X(q_m)) \\
        & =\dfrac{1}{\mathds{E}a^{Y_{m-1}}}\, \dfrac{1-aq_m}{1-q_m}\sum_{k=0}^\infty (1-q_m)(aq_m)^k \mathds{E}a^{Y_{m-1}}f(Y_{m-1}+k)\\
        &=\dfrac{\mathds{E}a^{Y_{m-1}}f(Y_{m-1}+X(aq_m))}{\mathds{E}a^{Y_{m-1}}},
    \end{split}
\end{equation*}
since $Y_{m-1}$ and $X(aq_m)$ are independent. Thus, the conclusion follows by induction on $m$.
\end{proof}

Denote
\begin{equation}\label{eq2.6}
    W_m(q)=X(q_1)+\cdots + X(q_m),\quad q_j=\dfrac{qj}{m},\quad j=1,\ldots, m,\quad 0<q<1.
\end{equation}
Various probabilistic representations of $S(n,m)$ in terms of the random variables considered in \eqref{eq2.2} and \eqref{eq2.6} are provided in the following result.

\begin{theorem}\label{th2}
For any $0<q<1$, we have
\begin{equation*}
    \begin{split}
      S(n,m) & =\dfrac{\mathds{E}S_m^{n-m}}{(n-m)!}=\dfrac{m^nq^{m-n}}{\prod_{j=1}^m (m-qj)}\dfrac{1}{2\pi} \int_{-\pi}^{\pi}\mathds{E}e^{i\theta(W_m(q)-n+m)}\, d\theta \\
        & =\dfrac{m^n}{m!} P\left ( X\left (1/m\right )+\cdots +X\left ((m-1)/m\right )\leq n-m\right ),
    \end{split}
\end{equation*}
where the last equality holds for $m\geq 2$.
\end{theorem}

\begin{proof}
The first equality was shown in \cite[Thm. 3.1]{AdeLek2019}. We give here a short proof of it for the sake of completeness. Starting from \eqref{eq1.i} and using \eqref{eq2.1} and \eqref{eq2.2}, we have for $|z|<1/m$
\begin{equation*}
    \sum_{n=m}^\infty S(n,m)z^n=z^m\mathds{E}e^{zS_m}=z^m \sum_{j=0}^\infty \dfrac{\mathds{E}S_m^j}{j!}z^j=\sum_{n=m}^\infty \dfrac{\mathds{E}S_m^{n-m}}{(n-m)!}z^n.
\end{equation*}
It therefore suffices to equate coefficients in this expression.

On the other hand, applying \eqref{eq1.i} with $z=qe^{i\theta}/m$ and recalling \eqref{eq2.4} and \eqref{eq2.6}, we see that
\begin{equation*}
    \begin{split}
        & \sum_{k=m}^\infty S(k,m)\left ( \dfrac{q}{m}\right )^ke^{i\theta k} = \left ( \dfrac{q}{m} \right )^m \dfrac{e^{i\theta m}}{\displaystyle\prod_{j=1}^m \left ( 1- qj/m\right )} \prod_{j=1}^m \dfrac{1-qj/m}{1-qje^{i\theta}/m}\\
         & = \left ( \dfrac{q}{m}\right )^m \dfrac{1}{\displaystyle\prod_{j=1}^m \left (1- qj/m \right)}\mathds{E}e^{i\theta (W_m(q)+m)}.
     \end{split}
\end{equation*}
Multiplying both sides of this identity by $e^{-i\theta n}/2\pi$ and then integrating over $[-\pi,\pi]$, we obtain the second equality in Theorem~\ref{th2}.

Finally, suppose that $m\geq 2$. Write \eqref{eq2.6} as
\begin{equation}\label{eq2.7}
    W_m(q)=Y_{m-1}+X(q),\qquad Y_{m-1}=X(q_1)+\cdots +X(q_{m-1}).
\end{equation}
Since $W_m(q)$ is an $\mathds{N}_0$-valued random variable, we have (see, for instance, Petrov \cite[p. 15]{Pet1995})
\begin{equation}\label{eq2.8}
    \begin{split}
       & \dfrac{1}{2\pi} \int_{-\pi}^\pi \mathds{E}e^{i\theta (W_m(q)-n+m)}\, d\theta=P(W_m(q)=n-m) \\
        & =\sum_{k=0}^{n-m}P(Y_{m-1}=k)P(X(q)=n-m-k)\\
        &=pq^{n-m}\sum_{k=0}^{n-m}\left (\dfrac{1}{q} \right )^k P(Y_{m-1}=k)
         = pq^{n-m}\mathds{E}\left ( \dfrac{1}{q}\right )^{Y_{m-1}}1_{[0,n-m]}(Y_{m-1})\\
        &=pq^{n-m}\mathds{E}\left ( \dfrac{1}{q}\right )^{Y_{m-1}} P\left ( X\left (1/m\right )+\cdots +  X\left ((m-1)/m\right )\leq n-m \right ),
    \end{split}
\end{equation}
where the last equality follows from \eqref{eq2.7} and Lemma~\ref{l1} with $a=q^{-1}$ and $f=1_{[0,n-m]}$. On the other hand, it follows from \eqref{eq2.4} and \eqref{eq2.7} that
\begin{equation*}
    \mathds{E}\left ( \dfrac{1}{q}\right )^{Y_{m-1}}=\prod_{j=1}^{m-1}\dfrac{1-qj/m}{1-j/m}=\dfrac{1}{(m-1)!}\prod_{j=1}^{m-1}(m-qj).
\end{equation*}
This, together with \eqref{eq2.8} and the second equality in Theorem~\ref{th2}, shows the third one and completes the proof.
\end{proof}

For other probabilistic representations of $S(n,m)$, we refer the reader to \cite{AdeCar2021} and \cite{AdeLek2019}.

\section{Estimates in the non-central regions}\label{s3}

Let $S_m$ be the random variable defined in \eqref{eq2.2}. By \eqref{eq2.1}, we have
\begin{equation}\label{eq3.9}
\begin{split}
    &\nu_m:=\mathds{E}S_m=\sum_{j=1}^m j=\dfrac{m(m+1)}{2},\\
    &\tau_m^2:=\text{Var}(S_m)=\sum_{j=1}^m j^2=\dfrac{m(m+1)(2m+1)}{6}.
\end{split}
\end{equation}

With these ingredients, our first main result is the following.

\begin{theorem}\label{th3}
Assume that
\begin{equation}\label{eq3.10}
    3\leq n-m\leq \dfrac{\nu_m}{2\tau_m}.
\end{equation}
Then,
\begin{equation*}
    \left | \dfrac{(n-m)!}{\nu_m^{n-m}}S(n,m)-1-\binom{n-m}{2} \dfrac{\tau_m^2}{\nu_m^2}\right |\leq 2e^{1/4}\left ( \dfrac{2(n-m)\tau_m}{\nu_m}\right )^3.
\end{equation*}
\end{theorem}

\begin{proof}
Consider the standardized random variable
\begin{equation*}
    Z_m=\dfrac{S_m-\nu_m}{\tau_m},
\end{equation*}
which obviously satisfies $\mathds{E}Z_m=0$ and $\text{Var}(Z_m)=\mathds{E}Z_m^2=1$. We write
\begin{equation}\label{eq3.11}
    \begin{split}
       & \mathds{E}S_m^{n-m}=\mathds{E}(\tau_mZ_m+\nu_m)^{n-m}=\sum_{j=0}^{n-m}\binom{n-m}{j}\tau_m^j \mathds{E}Z_m^j\nu_m^{n-m-j} \\
        & =\nu_m^{n-m} \left ( 1+\binom{n-m}{2}\dfrac{\tau_m^2}{\nu_m^2}+\sum_{j=3}^{n-m} \binom{n-m}{j}\left ( \dfrac{\tau_m}{\nu_m}\right )^j\mathds{E}Z_m^j\right ).
    \end{split}
\end{equation}
By assumption \eqref{eq3.10}, the modulus of the sum on the right-hand side in \eqref{eq3.11} is bounded above by
\begin{equation}\label{eq3.12}
\begin{split}
    &\sum_{j=3}^{n-m}\left ( \dfrac{2(n-m)\tau_m}{\nu_m}\right )^j \dfrac{\mathds{E}(|Z_m|/2)^j}{j!}\leq \left (\dfrac{2(n-m)\tau_m}{\nu_m}\right)^3\mathds{E}e^{|Z_m|/2}\\
    &\leq \left (\dfrac{2(n-m)\tau_m}{\nu_m}\right )^3\left ( \mathds{E}e^{Z_m/2}+\mathds{E}e^{-Z_m/2}\right ).
\end{split}
\end{equation}

Using the identity
\begin{equation*}
    -\log (1-z)-z=z^2g(z),\quad |g(z)|\leq 1,\quad |z|\leq 1/2,
\end{equation*}
we have from \eqref{eq2.1}, \eqref{eq2.2}, and \eqref{eq3.9}
\begin{equation}\label{eq3.13}
    \begin{split}
    &\mathds{E}e^{Z_m/2}=\prod_{j=1}^m \exp \left ( \dfrac{j}{2\tau_m}(T_j-1)\right )\\
    &=\prod_{j=1}^m \exp \left ( -\log \left ( 1-\dfrac{j}{2\tau_m}\right )-\dfrac{j}{2\tau_m}\right )\leq \prod_{j=1}^m \exp \left ( \dfrac{j^2}{4\tau_m^2}\right )=e^{1/4}.
    \end{split}
\end{equation}
Similarly,
\begin{equation*}
    \mathds{E}e^{-Z_m/2}\leq e^{1/4}.
\end{equation*}
In view of the first equality in Theorem~\ref{th2}, the conclusion follows from \eqref{eq3.11}--\eqref{eq3.13}.
\end{proof}

The assumption $n-m\geq 3$ in \eqref{eq3.10} is not essential. Actually, if $n-m=1$ or $n-m=2$, the first equality in Theorem~\ref{th2} and \eqref{eq3.11} give us, respectively,
\begin{equation*}
    \dfrac{S(n,n-1)}{\nu_{n-1}}=1,\quad \dfrac{2}{\nu_{n-2}^2}S(n,n-2)=1+\dfrac{\tau_{n-2}^2}{\nu_{n-2}^2}.
\end{equation*}

Recall that the generalized harmonic numbers are defined as
\begin{equation}\label{eq3.13star}
    H_{m,r}=\sum_{j=1}^m \dfrac{1}{j^r},\quad m,r\in \mathds{N}.
\end{equation}
We simply denote by $H_m:=H_{m,1}$ the harmonic numbers. Our second main result is the following.

\begin{theorem}\label{th4}
We have
\begin{equation*}
    S(n,m)\leq \dfrac{m^n}{m!}.
\end{equation*}
In addition, if $m\geq 2$ and $n\geq mH_m$, then
\begin{equation}\label{eq3.14}
    \dfrac{m^n}{m!}\left ( 1-\dfrac{m^2H_{m,2}-mH_m}{(n-mH_m+1)^2}\right )\leq S(n,m)\leq \dfrac{m^n}{m!}.
\end{equation}
\end{theorem}

\begin{proof}
Since the first statement is an immediate consequence of the third equality in Theorem~\ref{th2}, we will only show the second one. In the setting of Theorem~\ref{th2}, denote
\begin{equation*}
    V_{m-1}=X\left ( 1/m\right )+\cdots +X\left ( (m-1)/m\right ),\quad m\geq 2.
\end{equation*}
By \eqref{eq2.4}, we see that
\begin{equation*}
    \mathds{E}V_{m-1}=\sum_{j=1}^{m-1}\dfrac{j/m}{1-j/m}=\sum_{j=1}^{m-1}\dfrac{m}{m-j}-m+1=mH_{m-1}-m+1,
\end{equation*}
as well as
\begin{equation}\label{eq3.14star}
\begin{split}
    &\text{Var}(V_{m-1})=\sum_{j=1}^{m-1}\dfrac{j/m}{(1-j/m)^2}\\
    &=\sum_{j=1}^{m-1}\dfrac{m^2}{(m-j)^2}-\sum_{j=1}^{m-1}\dfrac{m}{m-j}=m^2H_{m,2}-mH_m.
\end{split}
\end{equation}
Hence, by Chebyshev's inequality, we have
\begin{equation*}
    \begin{split}
        & P(V_{m-1}\geq n-m+1)=P(V_{m-1}-\mathds{E}V_{m-1}\geq n-mH_{m-1}) \\
         & \leq \dfrac{\text{Var}(V_{m-1})}{(n-mH_{m-1})^2}=\dfrac{m^2H_{m,2}-mH_m}{(n-mH_m+1)^2}.
     \end{split}
\end{equation*}
Thus, the conclusion follows from the third equality in Theorem~\ref{th2}.
\end{proof}

\section{Estimates in the central region}\label{s4}

In this section, we assume that $m\geq 2$. Looking at \eqref{eq2.6}, we choose $q=q(n,m)\in (0,1)$ the unique solution to the equation
\begin{equation}\label{eq4.15}
    n-m=\mathds{E}W_m(q)=\sum_{j=1}^m \dfrac{qj/m}{1-qj/m}=\sum_{j=1}^{m-1}\dfrac{qj/m}{1-qj/m}+\dfrac{q}{1-q},
\end{equation}
as follows from \eqref{eq2.4}. Note that this failure probability $q=q(n,m)$ is well defined, because the function $q\to \mathds{E}W_m(q)$, $0<q<1$, is strictly increasing and goes from 0 to $\infty$. With this choice, we denote
\begin{equation}\label{eq4.16}
    W_m(q)=W_{m-1}(q)+X(q),\quad W_{m-1}(q)=X(q_1)+\cdots + X(q_{m-1}).
\end{equation}
The variances of such random variables are denoted by
\begin{equation}\label{eq4.17}
    \sigma_m^2(q)=\text{Var} (W_m(q)),\quad \sigma_{m-1}^2=\text{Var}(W_{m-1}(q)).
\end{equation}
By \eqref{eq2.4} and \eqref{eq4.16}, we see that
\begin{equation}\label{eq4.18}
    \sigma_m^2=\sigma_{m-1}^2+\dfrac{q}{(1-q)^2},\quad \sigma_{m-1}^2=\sum_{j=1}^{m-1}\dfrac{qj/m}{(1-qj/m)^2}.
\end{equation}

With these notations, we state our third main result.

\begin{theorem}\label{th5}
Let $q$ be as in \eqref{eq4.15} and $p=1-q$. If $m\geq 2$, then
\begin{equation*}
    \left | \dfrac{\prod_{j=1}^m (m-qj)}{m^nq^{n-m}}S(n,m)-\dfrac{1}{\sigma_m \sqrt{2\pi}}\right | \leq \dfrac{2+9\sqrt{2\pi}}{2\pi p\sigma_m^2}+ \dfrac{6\sqrt{2}}{\pi p \sigma_{m-1}\sigma_m}.
\end{equation*}
\end{theorem}

In Lemma~\ref{l9} below, we discuss the solution $q$ to equation \eqref{eq4.15} and give sharp estimates of $\sigma_{m-1}$. The applicability of Theorem~\ref{th5} is considered in Cases~3 and 4 in the following section.

The proof of Theorem~\ref{th5} is based upon the following three auxiliary results. In the first place, for any $k\in \mathds{N}$, let $\beta_k$ be a random variable having the beta density $\rho_k(x)=k(1-x)^{k-1}$, $0\leq x\leq 1$. The Taylor's formula with remainder in integral form can be written as
\begin{equation}\label{eq4.19}
    e^{z}=\sum_{j=0}^{k-1}\dfrac{z^j}{j!}+ \dfrac{z^k}{k!}\mathds{E}e^{z\beta_k},\quad k\in \mathds{N}.
\end{equation}
We always assume that any random variable $\beta_k$ is independent of all other random variables appearing in a same formula.

\begin{lemma}\label{l6}
In the setting of Theorem~\ref{th5}, we have for any $\theta\in \mathds{R}$
\begin{equation*}
    \mathds{E}e^{i\theta (W_m(q)-\mathds{E}W_m(q))}=\exp \left ( -\dfrac{\sigma_m^2 \theta^2 }{2}-ia_m(\theta)\theta^3\right ),
\end{equation*}
where
\begin{equation}\label{eq4.19.2}
    |a_m(\theta)|\leq \dfrac{\sigma_m^2}{3p}.
\end{equation}
\end{lemma}

\begin{proof}
Let $q_j=qj/m$ and $p_j=1-q_j$, $j=1,\ldots ,m$. From \eqref{eq2.4}, we have
\begin{equation}\label{eq4.20}
    \mathds{E}e^{i\theta (X(q_j)-\mathds{E}X(q_j))}=\exp \left ( -\log (1-q_je^{i\theta})+\log (1-q_j)-i\theta \dfrac{q_j}{1-q_j}\right ).
\end{equation}
With the help of \eqref{eq4.19}, we write the exponent in \eqref{eq4.20} as
\begin{equation}\label{eq4.21}
    \begin{split}
       & \sum_{k=1}^\infty \dfrac{q_j^k}{k} (e^{i\theta k}-1)-i\theta \sum_{k=1}^\infty q_j^k=\dfrac{1}{p_j} \sum_{k=1}^\infty p_jq_j^k \dfrac{e^{i\theta k}-1-i\theta k}{k} \\
        & =\dfrac{1}{p_j}\mathds{E} \left (\dfrac{e^{i\theta X(q_j)}-1-i\theta X(q_j)}{X(q_j)}\right )\\
        &=\dfrac{1}{p_j}\left ( -\dfrac{\theta^2}{2}\mathds{E}X(q_j)- i\dfrac{\theta^3}{6}\mathds{E}(X(q_j))^2e^{i\theta X(q_j)\beta_3}\right )\\
        &=- \dfrac{\theta^2}{2}\dfrac{q_j}{p_j^2}- i\theta^3 \dfrac{1}{6p_j}\mathds{E}(X(q_j))^2e^{i\theta X(q_j)\beta_3}.
    \end{split}
\end{equation}
By \eqref{eq4.17}, \eqref{eq4.20}, and \eqref{eq4.21}, we get
\begin{equation}\label{eq4.22}
\begin{split}
    &\mathds{E}e^{i\theta (W_m(q)-\mathds{E}W_m(q))}=\prod_{j=1}^m \mathds{E}e^{i\theta (X(q_j)-\mathds{E}X(q_j))}\\
    &=\exp \left ( -\dfrac{\theta^2}{2}\sigma_m^2-i\theta^3 a_m(\theta)\right ),
\end{split}
\end{equation}
where
\begin{equation*}
    a_m(\theta)=\dfrac{1}{6} \sum_{j=1}^m \dfrac{1}{p_j}\mathds{E}(X(q_j))^2 e^{i\theta X(q_j)\beta_3}.
\end{equation*}
Since $p_j\geq p_m=p$, $j=1,\ldots ,m$, we have from \eqref{eq2.4} and \eqref{eq4.18}
\begin{equation*}
    \begin{split}
       & |a_m(\theta)|\leq \dfrac{1}{6p}\sum_{j=1}^m \mathds{E}(X(q_j))^2=\dfrac{1}{6p}\sum_{j=1}^m \left ( \text{Var}(X(q_j))+(\mathds{E}X(q_j))^2\right ) \\
        & =\dfrac{1}{6p}\sum_{j=1}^m \left ( \dfrac{q_j}{p_j^2}+\dfrac{q_j^2}{p_j^2}\right )\leq \dfrac{\sigma_m^2}{3p}.
    \end{split}
\end{equation*}
This, together with \eqref{eq4.22}, shows the result.
\end{proof}

Let $Z$ be a random variable having the standard normal density
\begin{equation}\label{eq4.23}
    \rho (u)=\dfrac{1}{\sqrt{2\pi}}e^{-u^2/2},\quad u\in \mathds{R}.
\end{equation}
The even moments of $Z$ are given by
\begin{equation}\label{eq4.24}
    \mathds{E}Z^{2k}=\dfrac{(2k)!}{k! 2^k},\quad k\in \mathds{N}_0.
\end{equation}

\begin{lemma}\label{l7}
In the setting of Theorem~\ref{th5}, we have
\begin{equation*}
    \left | \dfrac{1}{2\pi}\int_{|\theta|\leq p} \mathds{E}e^{i\theta(W_m(q)-\mathds{E}W_m(q))}\, d\theta- \dfrac{1}{\sigma_m\sqrt{2\pi}}\right | \leq \dfrac{2+9\sqrt{2\pi}}{2\pi p \sigma_m^2}.
\end{equation*}
\end{lemma}

\begin{proof}
By \eqref{eq4.23} and Lemma~\ref{l6}, we have
\begin{equation}\label{eq4.25}
    \begin{split}
       & \int_{|\theta|\leq p}\mathds{E}e^{i\theta (W_m(q)-\mathds{E}W_m(q))}\, d\theta- \dfrac{\sqrt{2\pi}}{\sigma_m}\\
       &=\int_{|\theta|\leq p}e^{-\sigma_m^2\theta^2/2-ia_m(\theta)\theta^3}\, d \theta - \int_{-\infty}^\infty e^{-\sigma_m^2 \theta^2/2}\, d\theta\\
       &= \int_{|\theta|>p}e^{-\sigma_m^2 \theta^2/2}\, d\theta+ \int_{|\theta|\leq p} e^{-\sigma_m^2 \theta^2/2} \left ( e^{-ia_m(\theta)\theta^3}-1\right )\, d\theta:=I+II.
    \end{split}
\end{equation}
Observe that
\begin{equation}\label{eq4.26}
    I\leq \dfrac{2}{p}\int_p^\infty \theta e^{-\sigma_m^2 \theta^2/2}\, d\theta=\dfrac{2}{p\sigma_m^2}e^{-(p\sigma_m)^2/2}\leq \dfrac{2}{p\sigma_m^2}.
\end{equation}
On the other hand, using the inequality $|e^z-1|\leq |z|e^{|z|}$, we have from \eqref{eq4.19.2}
\begin{equation*}
    \begin{split}
        & |II|\leq \dfrac{\sigma_m^2}{3p}\int_{|\theta|\leq p}|\theta|^3 e^{-\sigma_m^2 \theta^2 /2}\, d\theta\leq \dfrac{3\sqrt{2\pi}}{p\sigma_m^2} \mathds{E}|Z|^3 \\
         & \leq\dfrac{3\sqrt{2\pi}}{p\sigma_m^2} (\mathds{E}Z^4)^{3/4}=\dfrac{3^{7/4}\sqrt{2\pi}}{p\sigma_m^2}\leq \dfrac{9\sqrt{2\pi}}{p\sigma_m^2}.
     \end{split}
\end{equation*}
where we have used \eqref{eq4.24} and Hölder's inequality. This, in conjunction with \eqref{eq4.25} and \eqref{eq4.26}, shows the result.
\end{proof}

In the following auxiliary result, we will need the identity
\begin{equation}\label{eq4.27}
    \left ( \sum_{j=1}^m x_j\right )^2=\sum_{j=1}^m x_j^2+2\sum_{1\leq i <j\leq m}x_ix_j,\quad x_1,\ldots ,x_m\in \mathds{R},
\end{equation}
as well as the inequality
\begin{equation}\label{eq4.28}
    1-\cos \theta\geq \dfrac{\theta^2}{2}\left ( 1-\dfrac{\theta^2}{12}\right )\geq \dfrac{\theta^2}{2}\left ( 1-\dfrac{\pi^2}{12}\right )\geq \dfrac{\theta^2}{12},\quad |\theta|\leq \pi.
\end{equation}

\begin{lemma}\label{l8}
In the setting of Theorem~\ref{th5}, we have
\begin{equation*}
    \dfrac{1}{2\pi}\int_{p\leq |\theta|\leq \pi}\left | \mathds{E}e^{i\theta (W_m(q)-\mathds{E}W_m(q))}\right |\, d\theta\leq \dfrac{6\sqrt{2}}{\pi p \sigma_{m-1}\sigma_m}.
\end{equation*}
\end{lemma}

\begin{proof}
We first claim that
\begin{equation}\label{eq4.29}
    \left | \mathds{E}e^{i\theta (W_m(q)-\mathds{E}W_m(q))}\right |= \prod_{j=1}^m \dfrac{1}{\sqrt{1+2(1-\cos\theta)q_j/p_j^2}}.
\end{equation}
In fact, we have from \eqref{eq2.4}
\begin{equation*}
    \left | \mathds{E}e^{i\theta (X(q_j)-\mathds{E}X(q_j))}\right |^2=\left | \mathds{E}e^{i\theta X(q_j)}\right |^2=\dfrac{p_j^2}{|1-q_je^{i\theta}|^2}=\dfrac{1}{1-2(1-\cos \theta)q_j/p_j^2},
\end{equation*}
which implies \eqref{eq4.29}. On the other hand, choosing $x_j=q_j/p_j^2$ in \eqref{eq4.27} and recalling \eqref{eq4.18}, we see that
\begin{equation*}
    \sigma_m^4\leq \dfrac{q}{p^2}\sigma_m^2+2\sum_{1\leq i<j\leq m}\dfrac{q_i}{p_i^2}\dfrac{q_j}{p_j^2},
\end{equation*}
since $q_j/p_j^2\leq q/p^2$, $j=1,\ldots ,m$. We therefore have from \eqref{eq4.18}
\begin{equation*}
    \begin{split}
        & \prod_{j=1}^m (1+2(1-\cos \theta)q_j/p_j^2)\geq 4(1-\cos \theta)^2\sum_{1\leq i <j\leq m}\dfrac{q_i}{p_i^2}\dfrac{q_j}{p_j^2} \\
         & \geq 2(1-\cos \theta)^2 \sigma_m^2 \left ( \sigma_m^2-\dfrac{q}{p^2}\right )= 2(1-\cos \theta)^2 \sigma_m^2 \sigma_{m-1}^2 \geq 2 \left ( \dfrac{\theta^2 \sigma_m \sigma_{m-1}}{12}\right )^2,
\end{split}
\end{equation*}
where the last inequality follows from \eqref{eq4.28}. By \eqref{eq4.29}, this implies that
\begin{equation*}
    \int_{p\leq |\theta|\leq \pi} \left | \mathds{E}e^{i\theta (W_m(q)-\mathds{E}W_m(q))}\right |\, d\theta\leq \dfrac{24}{\sigma_{m-1}\sigma_m \sqrt{2}}\int_p^\pi \dfrac{1}{\theta^2}\,d\theta\leq \dfrac{12\sqrt{2}}{p\sigma_{m-1}\sigma_m}.
\end{equation*}
The proof is complete.
\end{proof}

\begin{proof}[Proof of Theorem~\ref{th5}]
In view of the second equality in Theorem~\ref{th2}, Theorem~\ref{th5} is an immediate consequence of Lemmas~\ref{l7} and \ref{l8}.
\end{proof}

\section{Discussion}\label{s5}

In this section, we discuss which one of the Theorems~\ref{th3}, \ref{th4}, or \ref{th5} should be applied to estimate $S(n,m)$, as $n\to \infty$, depending on the values of $m$. To this end, we will need the following auxiliary result.

\begin{lemma}\label{l9}
Let $q$ and $\sigma_{m-1}^2$ be as in \eqref{eq4.15} and \eqref{eq4.17}, respectively, and let $p=1-q$. Then,
\begin{equation}\label{eq5.30}
    \dfrac{1}{mp}-\dfrac{1}{q}\log \left ( 1-q\dfrac{m-1}{m}\right )\leq \dfrac{n}{m}\leq \dfrac{1}{mp}-\dfrac{1}{q}\log (1-q)
\end{equation}
and
\begin{equation}\label{eq5.31}
    \dfrac{m(m-1)}{mp+q}+\dfrac{m}{q}\log \left ( \dfrac{(m-1)p+1}{m}\right )\leq \sigma_{m-1}^2\leq \dfrac{m}{p}+\dfrac{m}{q}\log \left ( \dfrac{mp}{m-q}\right ).
\end{equation}
\end{lemma}

\begin{proof}
By definition \eqref{eq4.15}, we have
\begin{equation*}
    n-m=\sum_{j=1}^m \dfrac{1}{1-qj/m}-m,
\end{equation*}
thus implying that
\begin{equation}\label{eq5.32}
    \dfrac{n}{m}=\dfrac{1}{mp}+\dfrac{1}{m}\sum_{j=1}^{m-1}\dfrac{1}{1-qj/m}.
\end{equation}
Since the function $f(x)=(1-qx/m)^{-1}$, $0\leq x\leq m$, is increasing, we get
\begin{equation*}
    \begin{split}
        & -\dfrac{m}{q}\log \left ( 1-q\dfrac{m-1}{m}\right )=\int_0^{m-1}f(x)\, dx\leq \sum_{j=1}^{m-1}\dfrac{1}{1-qj/m} \\
         & \leq \int_1^m f(x)\, dx=-\dfrac{m}{q}\log \left ( 1- q\dfrac{m-1}{m-q}\right ),
     \end{split}
\end{equation*}
which, in conjunction with \eqref{eq5.32}, implies \eqref{eq5.30}. On the other hand, the function
\begin{equation*}
    g(x)=\dfrac{qx/m}{(1-qx/m)^2}=\dfrac{1}{(1-qx/m)^2}-\dfrac{1}{1-qx/m},\quad 0\leq x\leq m,
\end{equation*}
is increasing as well. Therefore,
\begin{equation*}
    \begin{split}
        & \dfrac{m(m-1)}{mp+q}+\dfrac{m}{q}\log \left (1-q\dfrac{m-1}{m} \right )= \int_0^{m-1}g(x)\, dx\leq \sum_{j=1}^{m-1}\dfrac{qj/m}{(1-qj/m)^2} \\
         & \leq \int_1^m g(x)\, dx= \dfrac{n(m-1)}{(m-q)p}+\dfrac{m}{q}\log \left ( 1-q\dfrac{m-1}{m-q}\right ).
     \end{split}
\end{equation*}
In view of \eqref{eq4.18}, this implies \eqref{eq5.31} and completes the proof.
\end{proof}

In the first place, note that if $m$ is bounded then Theorem~\ref{th4} gives us an estimate of
\begin{equation}\label{eq5.33}
    \dfrac{m!}{m^n} S(n,m)
\end{equation}
with main term 1 and remainder term of the order of $1/n^2$.

From now on, assume that $m$ tends to infinity. We use the notation $x_m\sim y_m$ if
\begin{equation*}
    0<\varliminf_{m\to \infty}\dfrac{x_m}{y_m}\leq \varlimsup_{m\to \infty} \dfrac{x_m}{y_m}<\infty.
\end{equation*}
The following two cases are concerned with estimates in the non-central regions.

\noindent \textbf{Case 1}: $n-m\leq m^{\alpha}$, for some $0<\alpha<1/2$. Under this assumption, condition \eqref{eq3.10} is fulfilled, since $\nu_m/\tau_m\sim m^{1/2}$. Therefore, we apply Theorem~\ref{th3} to estimate
    \begin{equation*}
        \dfrac{(n-m)!}{\nu_m^{n-m}}S(n,m)
    \end{equation*}
    with main term 1 and remainder term of the order of
    \begin{equation*}
        \left (\dfrac{(n-m)\tau_m}{\nu_m}\right )^3\sim \left ( \dfrac{n-m}{\sqrt{m}}\right )^3.
    \end{equation*}

More generally, suppose that $n-m=o(\sqrt{m})$ or, equivalently, $n-m=o(\sqrt{n})$. In such a case, the first equality in Theorem~\ref{th2} and formula \eqref{eq3.11} allows us to give the exact expression
\begin{equation}\label{eq5.33.3}
\begin{split}
    &\dfrac{(n-m)!}{\nu_{m}^{n-m}}S(n,m)=1+\binom{n-m}{2} \left ( \dfrac{\tau_m}{\nu_m}\right )^2\\
    &+\sum_{j=3}^{n-m}\binom{n-m}{j}\left ( \dfrac{\tau_m}{\nu_m}\right )^j\mathds{E}Z_m^j.
    \end{split}
\end{equation}
This is essentially formula \eqref{eq1.f} shown by Moser and Wyman \cite{MosWym1958}. Note that the $j$-term in \eqref{eq5.33.3} has a coefficient depending on $\mathds{E}Z_m^j$ and order of magnitude $(o(\sqrt{m})/\sqrt{m})^j$.

\noindent \textbf{Case 2}: $n-m\geq m\log (mR_m)$, with $R_m\to \infty$, as $m\to \infty$. In this case, we use Theorem~\ref{th4} to estimate \eqref{eq5.33} with main term 1 and remainder term of the order of
\begin{equation}\label{eq5.33.2}
    \left ( \dfrac{m}{n-mH_m} \right )^2
\end{equation}
Since $H_m\sim \log m$, the order of magnitude of \eqref{eq5.33.2} is at least $(\log R_m)^{-2}$.

In the following two cases, we apply Theorem~\ref{th5} to give estimates in the central region.

\noindent \textbf{Case 3}: $m^\alpha\leq n-m\leq m\log (mr_m)$, where $0<\alpha <1/2$, and $r_m\to 0$ and $mr_m\to \infty$, as $m\to \infty$. Equivalently,
\begin{equation}\label{eq5.34}
    1+\dfrac{1}{m^{1-\alpha}}\leq \dfrac{n}{m}\leq 1+ \log (mr_m).
\end{equation}
In view of \eqref{eq5.30}, we consider the increasing function
\begin{equation*}
    \varphi (q)=\dfrac{1}{mp}-\dfrac{1}{q}\log (1-q),\quad 0<q<1.
\end{equation*}
Note that
\begin{equation*}
    \varphi (q)\sim 1+\dfrac{q}{2},\ q\to 0,\quad \varphi(q) \sim \dfrac{1}{mp}-\log (1-q),\ q\to 1,
\end{equation*}
thus implying that
\begin{equation}\label{eq5.35}
\begin{split}
    &\varphi \left ( \dfrac{2}{m^{1-\alpha}}\right )\sim 1+\dfrac{1}{m^{1-\alpha}},\\
    &\varphi \left ( 1-\dfrac{1}{mr_m}\right )\sim r_m+\log (mr_m)\sim\log (mr_m),
\end{split}
\end{equation}
thanks to the assumptions on $r_m$. As a consequence of \eqref{eq5.34} and \eqref{eq5.35}, we see that
\begin{equation}\label{eq5.36}
    \dfrac{1}{mr_m}\lesssim p \lesssim 1-\dfrac{2}{m^{1-\alpha}}.
\end{equation}
Let us consider the extreme cases in \eqref{eq5.36}. Suppose that $p\sim1/mr_m$. From \eqref{eq4.18} and \eqref{eq5.31}, we have
\begin{equation*}
    \sigma_{m-1}^2\sim m^2 r_m,\quad \sigma_m^2\sim m^2 r_m^2+m^2 r_m\sim m^2 r_m\sim \sigma_{m-1}^2.
\end{equation*}
Therefore, we apply Theorem~\ref{th5} to estimate
\begin{equation}\label{eq5.37}
    \dfrac{\prod_{j=1}^m (m-qj)}{m^nq^{n-m}}S(n,m)
\end{equation}
with main term of the order of $1/\sigma_m\sim 1/m\sqrt{r_m}$ and a remainder term of the order of
\begin{equation*}
    \dfrac{1}{p\sigma_m^2}\sim \dfrac{1}{p\sigma_m\sigma_{m-1}}\sim \dfrac{1}{m}.
\end{equation*}
Suppose that $p\sim 1- 2/m^{1-\alpha}$. Again by \eqref{eq4.18} and \eqref{eq5.31}, we have
\begin{equation*}
    \sigma_m^2\sim\sigma_{m-1}^2 \sim m.
\end{equation*}
Thus, we use again Theorem~\ref{th5} to estimate \eqref{eq5.37} with a main term of the order of $1/\sigma_m\sim 1/\sqrt{m}$ and a remainder terms of the order of
\begin{equation*}
    \dfrac{1}{p\sigma_m^2}\sim\dfrac{1}{p\sigma_m\sigma_{m-1}}\sim \dfrac{1}{m}.
\end{equation*}
Finally, for $p$ in the range given by \eqref{eq5.36}, we have a similar discussion: we use Theorem~\ref{th5} to estimate \eqref{eq5.37}. The order of the main term may change, but the remainder term has always the order of magnitude of $1/m$.

\noindent \textbf{Case 4}: $n-m\sim m\log (mr_m)$, where $r_m\sim r$, for some $0<r<\infty$. Proceeding as in Case~3, we have
\begin{equation*}
    p\sim\dfrac{1}{m},\quad \sigma_m^2 \sim \sigma_{m-1}^2\sim m^2.
\end{equation*}
We apply Theorem~\ref{th5} to estimate \eqref{eq5.37} with the following difference: the main term and the remainder term have the same order of magnitude $1/m$. As a consequence, Theorem~\ref{th5} only gives us in this case upper and lower bounds for \eqref{eq5.37} having the same order of magnitude, since
\begin{equation*}
    \dfrac{1}{\sigma_m}\sim \dfrac{1}{\sigma_{m-1}}\sim \dfrac{1}{p\sigma_m^2}\sim \dfrac{1}{p\sigma_m\sigma_{m-1}}\sim \dfrac{1}{m}.
\end{equation*}

\noindent \textbf{Funding}

This work was supported by Ministerio de Ciencia, Innovaci\'{o}n y Universidades, Project PGC2018-097621-B-100.













\vspace{1cc}


{\small
\noindent

}\end{document}